\documentclass[12pt]{iopart}
\usepackage{ntheorem}
\usepackage{enumerate}
 \usepackage{ragged2e}
 \usepackage{caption}
  \usepackage{color}
\usepackage{hyperref}
\hypersetup{
    linkcolor=blue,
    filecolor=gray,
    urlcolor=blue,
    citecolor=blue,
}

\usepackage{amssymb}
  \theoremsymbol{\mbox{$\blacksquare$}}
\usepackage{euscript,eufrak,verbatim}
\usepackage[dvips]{graphicx}
\usepackage[final]{epsfig}
\usepackage{url}

\newtheorem{example}{Example}

\newtheorem{lemma}{Lemma}
\newtheorem{proposition}{Proposition}

\newtheorem*{proof}{Proof}
\newtheorem{definition}{Definition}
\newtheorem{remark}{Remark}

\renewcommand{\theequation}{\thesection .\ \arabic{equation}}

\begin{document}

\title{$\alpha$-limit sets and Lyapunov function for maps with one topological attractor}

\author{ Yiming Ding \& Yun Sun}

\address{  Center for Mathematical Sciences and Department of Mathematics, Wuhan University of Technology, Wuhan 430070, China}

\ead{dingym@whut.edu.cn and sunyun@whut.edu.cn }
\vspace{10pt}
\begin{indented}
\item[December 2020]
\end{indented}
\begin{abstract}
We consider the topological behaviors of continuous maps with one topological attractor on compact metric space $X$. This kind of map is a generalization of maps such as topologically expansive Lorenz map, unimodal map without homtervals and so on. We provide a leveled $A$-$R$ pair decomposition for such maps, and characterize $\alpha$-limit set of each point. Based on weak Morse decomposition of $X$, we construct a bounded Lyapunov function $V(x)$, which give a clear description of orbit behavior of each point in $X$ except a meager set.
\end{abstract}

\section{Introduction}  
\par Let $(X,d)$ be a compact metric space and $f:X\to X$ be a continuous map. Suppose $Y\subseteq X$, $Y$ is invariant if $f(Y)= Y$,  $Y$ is forward invariant if $f(Y)\subseteq Y$, $Y$ is backward invariant if $f^{-1}(Y)\subseteq Y$, and it is bi-invariant if $Y$ is forward invariant and backward invaraint. $\forall x\in X$, $orb(x)$ is the orbit with initial value $x$, $\omega(x)=\bigcap_{n\geq0}\overline{\bigcup_{k\geq n}f^{k}(x)}$ and $\alpha(x)=\bigcap_{n\geq0}\overline{\bigcup_{k\geq n}\{f^{-k}(x)\}}$. $f$ is topologically transitive on $X$ if there exists a point $x\in X$ such that the orbit of $x$ is dense in $X$. A point $x\in X $ is nonwandering for $f$ if for any neighborhood $U$ of $x$, there exists $n>0$ such that $f^{n}(U)\cap U\neq \emptyset$, we denote by $\Omega(f)$ the set of nonwandering points. Let \textsl{int}$(Y)$ be the interior of $Y$, if \textsl{int}$(\overline{Y})=\emptyset$, $Y$ is nowhere dense. $Y$ is a meager set if $Y$ is the union of countable nowhere dense sets. Denote $d(x,Y)$ as the distance between $x$ and $Y$.
\par Attractor is one of the main concepts in the theory of dynamical systems, was early studied in the literatures \cite{ref1,ref2,ref7,ref4,ref5,ref3}. Milnor \cite{ref6} put forward the definition of metric attractor. A metric attractor is an invariant compact subset $A$ such that its basin of attraction $\mathcal{B}(A)=\{x\in X:\omega(x)\subset A\}$ has positive Lebesgue measure, while any smaller invariant compact set does not have this property \cite{ref6}. Since then the metric properties of attractor have been studied extensively \cite{ref16,ref17,ref18}. Here we focus on topological attractors. According to Smale \cite{ref4}, topological attractor $A$ must have a neigborhood $U$ so that $A$ is equal to the intersection of the images $f^{m}(U)$ for $m>0$. A related but simpler definition was given in \cite{ref5}:``A subset $A$ of $\Omega(f)$ is an attractor of $f$, provided it is indecomposable and has a neighborhood $U$ such that $f(U)\subset U$ and $\bigcap_{i>0}f^{i}(U)=A$", where $A$ is indecomposable means it's not the union of two disjoint closed invariant subsets. Collet and Eckmann \cite{ref7} introduced the attracting set, and attractor is an attracting set with a dense orbit. Based on the ideas above, we give a definition of attracting set and topological attractor (see Definition 2.1).
\par Let $\mathcal{A}_{f}$ be the collection of all attracting sets of $f$. In this paper we consider maps $f\in\mathcal{F}$ which satisfy the following  two assumptions:
\par (1) $f$ has unique topological attractor $A\in\mathcal{A}_{f}$.
\par (2) $\lim_{n\to\infty}A_{n}=A$, where $\{A_{n}\}\subseteq \mathcal{A}_{f}$ and $A_{n-1}\supsetneqq A_{n}$ $(n\geq1)$.
 \par For any $f\in\mathcal{F}$, we consider $f$ with only one attractor since $\alpha$-limit sets could be complex when attractor is not unique. If the number of attracting sets of $f$ is infinite and $A_{n-1}\supsetneqq A_{n}$ for all $n\geq1$, then any infinite sequence $\{A_{n}\}\subseteq \mathcal{A}_{f}$ converges to $A$; if there are finite attracting sets of $f$, then any attracting set is not the limit set. The second assumption implies $\mathcal{A}_{f}$ has at most one limit set which can only be the attractor. For instance, topologically expansive Lorenz maps and unimodal maps without homtervals satisfy these two assumptions. It is interesting to know whether there exists $f$ with one topological attractor such that the limit sets of $\mathcal{A}_{f}$ are not unique.
\par We are concerned with the $\alpha$-limit sets of $f\in \mathcal{F}$.
A lot of studies have been done on $\omega(x)$ because it shows the target of the orbit of $x$, while $\alpha(x)$ may be regarded as the source of the orbit of $x$. In particular, some important dynamical concepts such as internally chain transitive sets \cite{ref19}, homoclinic orbit \cite{ref8}, asymptotic periodicity \cite{ref26}, Morse decomposition \cite{ref9} and fake link \cite{ref10}, are defined via $\alpha$-limit sets. The $\alpha$-limit set often refers to homeomorphisms because the inverse $f^{-1}(x)$ is just one point. For an endomorphism $f$, it seems that the $\alpha$-limit set is more difficult to understand than the $\omega$-limit set in general, because the pre-image $\{f^{-k}(x)\}$ is more complex than a forward iterate $f^{k}(x)$. But $f$ may not have so many different $\alpha$-limit sets because $\alpha$-limit set is large in some sense. $\alpha$-limit sets of topologically expansive Lorenz maps can be characterized via consecutive renormalization process \cite{ref11}, so it is for unimodal maps without homtervals \cite{ref12}. In fact, the minimal renormalization corresponds to the maximal attracting set, we shall adapt this idea to $f\in\mathcal{F}$ and characterize the $\alpha$-limit sets of a general continuous map.
\par Isolated invariant sets (see Definition 4.3) are singled out because they can be continued to nearby equations in a natural way, in this sense they are stable objects. Conley \cite{ref9} obtained the fundamental decomposition theorem of isolated invariant sets and extended it to Morse decomposition. Since then the Conley index theory had been studied extensively \cite{ref22,ref20,ref23,ref28}. Liu \cite{ref20} studied the Conley index for random dynamical systems, Wang \cite{ref28} studied the Conley index and shape Conley index in general metric spaces. Here we consider the map $f\in\mathcal{F}$ on the compact metric space $X$. Observe that proper attracting set and corresponding repelling set could be intersected, but the interior is always disjoint. In section 4, we prove that they are both isolated invariant sets when they are not intersected. Conley's work focused on flow or homeomorphism, we adapt the main ideas to a general continuous map and obtain a bounded Lyapunov function $V(x)$.
\par In this paper, we consider the topological behaviors of $f\in\mathcal{F}$ on compact metric space $X$. This kind of map is a generalization of maps such as topologically expansive Lorenz map, unimodal map without homtervals and so on. It is clear that $X$ is not an isolated invariant set, but $X$ may be decomposable when $f$ is topologically transitive on $X$. We first introduce the $A$-$R$ pair decomposition of $X$ and obtain the leveled $A$-$R$ pair decomposition of $X$, then use it to characterize $\alpha$-limit set of each point. Since the attracting set and corresponding repelling set may be intersected, we slightly modify the definition of Morse sets and obtain the weak Morse decomposition of $X$. Based on weak Morse decomposition, we construct a bounded Lyapunov function $V(x)$, which can give a clear description of orbit behavior of each point in $X$ except a meager set.
\par The paper is organized as follows. We introduce leveled $A$-$R$ decomposition of $X$ in Section 2, and characterize the $\alpha$-limit set of each point in Section 3. In Section 4 we introduce weak Morse sets and obtain a bounded Lyapunov function $V(x)$ via weak Morse decomposition of $X$. Although we characterize the topological properties of $\alpha$-limit sets of $f\in \mathcal{F}$, the metric properties of such sets are not clear. Under some regularity conditions, it is still unknown if the $\alpha$-limit set is of Lebesgue measure zero, and if the Bowen dimension formula holds. For maps with more than one attractor, even the topological characterizations of $\alpha$-limit sets are still unclear.

\section{Leveled $A$-$R$ Pair Decomposition}
 \par Conley \cite{ref9} obtained the attractor-repeller pair decomposition of isolated invariant set, but some sets may be decomposable even if they are not isolated invariant set, such as the whole space $X$. We first introduce attracting set and repelling set, and obtain the $A$-$R$ pair decomposition of $X$. Then we consider the continuous decomposition of maximal attracting set and obtain leveled $A$-$R$ pair decomposition of $X$.
\begin{definition} \rm
A compact invariant set $A\subseteq X$ will be called an attracting set if there exists an open set $U\subseteq X$ such that $\forall x\in U$, $\omega(x)\subseteq A$, and denote the set $\mathcal{B}(A)=\{x\in X, \omega(x)\subseteq A\}$ be the basin of attraction.  A topological attractor is an attracting set which contains a dense orbit.
\end{definition}
\par According to Definition 2.1, if $A$ is an attracting set, it is clear that \textsl{int}$(\mathcal{B}(A))$ is a nonempty open set, i.e., \textsl{int}$(\mathcal{B}(A))\neq\emptyset$. The dual repelling set of attracting set $A$ is
$$ R:=\{x\in X, orb(x)\cap\textsl{int}(\mathcal{B}(A))=\emptyset\}.
$$
Notice that $A$ and $R$ may be intersected, but the interior is always disjoint. Denote $L:=A\cap R$ and $(A,R)$ be an $A$-$R$ pair decomposition of $X$, then $$X=A\cup R\cup C,$$ where $C:=C(A,R)=\{x\in X| \omega(x)\subseteq A, \alpha(x)\subseteq R\}$ is the connecting orbit. In fact, $X$ can be divided into $R$ and $(X\setminus R)\cup L$. According to the definition of $R$, $F:=X\setminus R=\bigcup_{n\geq0}f^{-n}(\textsl{int}(\mathcal{B}(A)))$ and $C=X\setminus (R\cup A)=(F\cup L)\setminus A$, it is clear that $\omega(x)\subseteq A$ for all $x\in C$. According to Theorem A which will be proved in Section 3, $\alpha(x)\subseteq R$ for $\forall x\in C $.
\par Topological transitivity is a classical topological property. Suppose $A$ be the unique attractor, then $f$ confined on $A$ is topologically transitive, $A$ is indecomposable and $A$ contains a dense orbit. For any $f\in\mathcal{F}$, we first check whether it is topologically transitive on the $X$ or not. If $f$ is topologically transitive on $X$, then the whole compact space $X$ is indecomposable and itself is a topological attractor. In this case, $A=X$, $R=\emptyset$.
 \begin{lemma} \rm Suppose $f\in\mathcal{F}$, we have the following:
\par (1) If $f$ is not transitive on $X$, then it exists a maximal proper attracting set $A_{1}$.
\par (2) There exists an integer $m$ $(0\leq m\leq+\infty)$ such that we can get a cluster of proper attracting sets ordered as $ A_{1}\supset A_{2}\supset\cdots \supset A_{m}$.

 \end{lemma}
\begin{proof}\rm(1) If $f$ is not topologically transitive on $X$, denote $A$ be the unique topological attractor. According to Definition 2.1, $A$ is an attracting set. So the existence of attracting set is obtained, next we show these attracting sets are nested. Given any $A_{i},A_{j}\in \mathcal{A}_{f}$, if $A_{i},A_{j}$ are not nested, then they could be disjoint or intersected. For the disjoint case, it will lead to two attractors and for the intersected case, it will lead to three attractors, both two cases contradict with our initial condition on unique attractor. Now applying the second assumption that $\mathcal{A}_{f}$ has at most one limit set which can only be the attractor. Hence the maximal proper attracting set is obtained. Denote $A_{1}$ be the maximal proper attracting set when $f$ is not topologically transitive on $X$.
\par (2) If $f$ is not topologically transitive on $X$, we can obtain a maximal attracting set $A_{1}\subset X$. However, $f$ confined on $A_{1}$ is also a continuous map with one attractor, which means $f_{1}:=f|_{A_{1}}\in\mathcal{F}$ and $(A_{1},f_{1})$ is a dynamical system. If $f_{1}$ is topological transitive on $A_{1}$, then $A_{1}$ is the unique attractor of $f$. If $f_{1}$ is not topologically transitive on $A_{1}$, repeat the process in (1), there exists a proper maximal attracting set $A_{2}\subset A_{1}$, and $f_{2}$ confined on $A_{2}$ is also a map in $\mathcal{F}$. In this way, we can define $f_{m-1}$ is not topologically transitive on $A_{m-1}$ and obtain the attracting set $A_{m}\subset A_{m-1}$. Then we obtain a cluster of attracting sets ordered as: $X=A_{0}\supset A_{1}\supset A_{2}\supset\cdots \supset A_{m}$, and $m$ could be infinite. $\hfill\square$

\end{proof}

\par Denote $m$ be the level of dynamical system $(X,f)$ if the process above can proceed $m$ times consecutively. The level of $(X,f)$ can be used to understand the relationship between attractor and attracting set. Denote $A_{0}:=X$, if $m$ is finite, $A_{m}$ is the unique attractor since $A_{m}$ can not be decomposed; if $m$ is infinite, attractor is the intersection of all attracting sets.
\begin{definition} \rm
If the level of $(X,f)$ is $m$, since each attracting set corresponds a repelling set, there exist $m$ proper attracting sets and $m$ corresponding repelling sets. And the set of all $(A_{i}$,$R_{i})$ $(0\leq i\leq m)$ is called the leveled $A$-$R$ pair decomposition of $X$, denoted as $\{(A_{1},R_{1}), (A_{2},R_{2}),\cdots,(A_{m},R_{m})\}$.
\end{definition}

\begin{example}\rm
Let $f: I=[0,1]\to I$ be an $m$-renormalizable topologically expansive Lorenz map and $c$ be the critical point. We regard the critical point $c$ as two points $c_{+}$ and $c_{-}$, then $f$ is continuous in some sense since $f$ is topologically conjugate to a continuous map on symbolic space \cite{ref25}. According to \cite{ref11}, define $Rf$ to be the minimal renormalization map of $f$ and $[a_{1},b_{1}]$ be the renormalization interval, then $Rf$ confined on $[a_{1},b_{1}]$ is also an expansive Lorenz map. Suppose $E_{1}$ be the minimal proper completely invariant closed set corresponds to the minimal renormalization, i.e., $f(E)=f^{-1}(E)=E$. Put $$e_{1-}=\sup\{x\in E_{1},x<c\} ,\ \ \ \ e_{1+}=\inf\{x\in E_{1},x>c\},$$
the maximal attracting set $A_{1}$ is $orb[a_{1},b_{1}]$ and repelling set $E_{1}=I\setminus \bigcup_{n\geq0}f^{-n}(a,b)$. Then consider the dynamical system $([a_{1},b_{1}],Rf)$ and obtain $A_{2}$ and $E_{2}$. Repeat the process above $m$ times since $f$ is $m$-renormalizable and we can obtain $m$ attracting sets and $m$ repelling sets. Then the leveled $A$-$R$ pair decomposition of $I$ is $\{(A_{1},R_{1}), (A_{2},R_{2}),\cdots,(A_{m},R_{m})\}$.  Notice that if $m$ is finite and the endpoints of $m$th renormalization interval $a_{m}=e_{m-}$ or $b_{m}=e_{m+}$, then $A_{m}\cap E_{m}\neq \emptyset$, but their interior is always disjoint. According to \cite{ref27}, the repelling sets of piecewise linear Lorenz maps are more simple since all the renormalizations are periodic.
\end{example}

\section{$\alpha$-limit Set of each Point}

\par We prove proper repelling sets and proper $\alpha$-limit sets are the same sets from different aspects, and characterize the $\alpha$-limit set of each point via leveled $A$-$R$ pair decomposition of $X$. Notice that attracting set and repelling set may be intersected, hence we particularly describe the $\alpha$-limit sets of intersection points.
\begin{lemma}\rm
 Let $f\in \mathcal{F}$ and the level of $(X,f)$ is $m$, then the following statements hold:
\par (1) $\forall x\in X$, $\alpha(x)$ is a bi-invariant closed set of $f$.
\par (2) Each proper repelling set is a proper bi-invariant closed set.
\par (3) Each proper bi-invariant closed set is a proper $\alpha$-limit set.
\end{lemma}
\begin{proof}\rm (1) $ \forall x \in X$, $\alpha(x)=\bigcap_{n\geq0}\overline{\bigcup_{k\geq n}\{f^{-k}(x)\}}$. For each $n\in \mathbb{N}$, $\bigcup_{k\geq n}\{f^{-k}(x)\}$ is backward invariant under $f$, it follows $ f^{-1}(\alpha(x))\subseteq\alpha(x)$.
Remember $y\in\alpha(x)$ is equivalent to the fact that there exists a sequence $\{x_{k}\}\subset X$ and an increasing sequence $\{n_{k}\}\subset\mathbb{N}$ such that $f^{n_{k}}(x_{k})=x$ and $x_{k}\to y$ as $k\to \infty$. Assume $y\in \alpha(x)$, we have $f(y)\in\alpha(x)$. So we conclude $f(\alpha(x))\subseteq\alpha(x)$.
\par (2) If $f\in \mathcal{F}$ and the level of $(X,f)$ is $m$, then there exist $m$ proper attracting sets. Given a proper attracting set $A_{i}$ $(0<i\leq m)$, denote $F_{i}:=\{x\in X,orb(x)\cap \textsl{int}(\mathcal{B}(A_{i}))\neq\emptyset\}$, then $F_{i}=\bigcup_{n\geq0}f^{-n}(\textsl{int}(\mathcal{B}(A_{i})))$ and $F_{i}$ is a proper bi-invariant open set. According to the definition of repelling set, $$R_{i}= \{x\in X,orb(x)\cap \textsl{int}(\mathcal{B}(A_{i})))=\emptyset\}=X\setminus F_{i},$$ then $R_{i}$ is a proper bi-invariant closed set.
\par(3) According to Definition 2.3, the leveled $A$-$R$ decomposition of $X$ is $\{(A_{1},R_{1}),(A_{2},R_{2}),$ $\cdots,(A_{m},R_{m})\}$. Since $A_{1}\supset A_{2}\supset\cdots \supset A_{m}$ and each proper repelling set is a proper bi-invariant closed set, we have $m$ proper bi-invariant sets ordered as $R_{1}\subset R_{2}\subset\cdots \subset R_{m}$. Next we prove that $R_{i}$ is an $\alpha$-limit set of $f$ for $i=1,2,\cdots m$. Suppose $x\in R_{i}\setminus R_{i-1}$, it must have $\alpha(x)=R_{k}$ for some $k=1,2,\cdots m$, because $f$ admits exact $m$ proper bi-invariant closed sets. By the definitions of $R_{i-1},R_{i}$ and $R_{i+1}$, we can know that $R_{i-1}\subsetneqq\alpha(x)$ because $x\notin R_{i-1}$, and $x\in R_{i}$ implies $\alpha(x)\subseteq R_{i}$. Since $x\notin R_{i+1}\setminus R_{i}$, $\alpha(x)\subsetneqq R_{i+1}$. It follows that $R_{i-1}\subsetneqq\alpha(x)\subsetneqq R_{i+1}$. Hence we can obtain $\alpha(x)=R_{i}$ for $x\in R_{i}\setminus R_{i-1}$, and each proper bi-invariant closed set $R_{i}$ is an $\alpha$-limit set. $\hfill\square$
\end{proof}
\par Now we are ready to characterize the $\alpha$-limit set of each point. Denote $L_{i}:=A_{i}\cap R_{i}$ $(0\leq i\leq m)$ and $L:=\bigcup_{i=0}^{m}L_{i}$. In fact, each repelling set $R_{i}$ is nowhere dense, and $L_{i}\subset R_{i}$ indicates $L_{i}$ is also nowhere dense. Hence $L$ is a meager set.

\par {\bf Theorem A} \ \ Let $f\in \mathcal{F}$ and the level of $(X,f)$ be $m$ $(0\leq m\leq\infty)$, we have:
\par (1) $f$ admits $m$ proper $\alpha$-limit sets listed as
$$ R_{1}\subset R_{2}\subset\cdots \subset R_{m}\subset X.$$
\par (2) If $L=\emptyset$, $\alpha(x)=R_{i}$ $(0< i\leq m)$ if and only if $$x\in A_{i-1}\setminus A_{i},$$ and $\alpha(x)=X$ if and only if $$x\in A:=\bigcap_{i=0}^{m}A_{i}.$$

\begin{proof}\rm (1) By Lemma 3.1, we know that proper repelling set and $\alpha$-limit set of $f$ are the same sets in different aspects. If the level of $(X,f)$ is $m$, then $f$ has exact $m$ proper $\alpha$-limit sets. Follows from the proof of Lemma 3.1, all the $\alpha$-limit sets are defined as $ R_{i}=\{x\in X,orb(x)\cap \textsl{int}(\mathcal{B}(A_{i}))=\emptyset\}$, $i=1,2,\cdots m$, and\begin{center}
{$ R_{1}\subset R_{2}\subset\cdots \subset R_{m}\subset X .$
}
\end{center}
\par (2) If $L=\emptyset$, at first, we describe the set $\{x\in X,\alpha(x)=R_{1}\}$ where $R_{1}$ is the repelling set corresponding to the maximal attracting set $A_{1}$.
\par\textbf{Claim:} $\alpha(x)=R_{1}$ if and only if $x\in A_{0}\setminus A_{1}=X\setminus A_{1}$.
\par By the proof of (1), $R_{1}$ is the minimal proper $\alpha$-limit set of $f$. So $\alpha(x)\supseteq R_{1}$ for all $x\in X$. Let $D$ be the minimal bi-invariant closed set of $f$ confined on $A_{1}$ and $f_{i}:=f|_{A_{i}}$ $(0\leq i\leq m)$. It follows that $D\cap R_{1}=\emptyset$ and $ D\subset R_{2}$. If $x\in X\setminus A_{1}$, i.e., $x\notin A_{1}$, since $A_{1}$ is invariant under $f$, for any $n\in\mathbb{N}$ we have $f^{-n}(x)\cap A_{1}=\emptyset$. So $\alpha(x) \cap A_{1}=\emptyset$, which indicates $\alpha(x)\cap D=\emptyset$. Hence, $\alpha(x)\subsetneqq R_{2}$ and $\alpha(x)=R_{1}$.
\par On the other hand, for any $x\in A_{1}$, since $orb(x,f|_{1})=orb(x,f)\cap A_{1}$, we see that $\alpha(x)\supset \alpha(x,f|_{1})$. By the minimality of $D$, $\alpha(x,f|_{1})\supseteq D$ for all $x\in A_{1}$. So $\alpha(x)\cap D \neq\emptyset$, which implies that $\alpha(x)\neq R_{1}$ for $x\in A_{1}$. The proof of the Claim is completed.
\par For $0\leq i\leq m$, we denote $A_{i}$ as the $ith$ proper attracting set of $f$, and $R_{i}$ as the corresponding repelling set. By the Claim we know that $\alpha(x)=R_{1}$ if and only if
$$x\in X\setminus A_{1}=A_{0}\setminus A_{1}.$$

\par For the case $i=2\leq m$, we consider the map $f_{1}$. According to the Claim, we obtain that $\alpha(x,f_{1})=D$ if and only if $x\in A_{1}\setminus A_{2}$.  It
follows that $\alpha(x)=R_{2}$ if and only if
\begin{center}
{$x\in A_{1}\setminus A_{2}.$
}
\end{center}
\par Repeat the above arguments, we conclude $\alpha(x)=R_{i}$ if and only if
\begin{center}
{$x\in A_{i-1}\setminus A_{i},$\quad  $  {\rm for}\quad0<i\leq m.$
}
\end{center}
\par If $m<\infty$, $f_{m}$ is topologically transitive on $A_{m}$, $\alpha(x,f_{m})=A_{m}$ for all $x\in A_{m}$. Then we have $\alpha(x,f|_{m})\subset\alpha(x)$, and $A_{m}\subset \alpha(x)$ for all $x\in A_{m}$. If $\alpha(x)\neq X$ for any $x\in A_{m}$, then $X\setminus \alpha(x)$ is a bi-invariant open set since $\alpha(x)$ is proper bi-invariant closed, and there exists another attractor in $X\setminus \alpha(x)$, which contradicts with the unique attractor. Hence we can conclude that $\alpha(x)=X$ for all $x\in A_{m}$.
\par For the case $m=\infty$, put $A:=\bigcap_{i=1}^{\infty}A_{i}$. Similarly, $\alpha(x)=X$ for all $x\in A$. The proof of Theorem A is completed. $\hfill\square$

\end{proof}
\begin{remark} \rm If $L\neq\emptyset$, $\forall x\in L_{i}$, its $\alpha(x)$ depends on which set $x$ belongs to, i.e., $x\in A_{i}$ or $x\in R_{i}$. According to Theorem A, if regrad $x\in A_{i}$, $\alpha(x)=R_{i+1}$; if regard $x\in R_{i}$, $\alpha(x)=R_{i}$. In this way, the $\alpha$-limit sets of points in meager set $L$ are also characterized.
\end{remark}

\par We can determine the $\alpha$-limit set of each point. If $m$=0, then $f$ is topologically transitive on $X$ which implies $\alpha(x)=X$ for any $x\in X$. Since $X$ is the largest $\alpha$-limit set, $f$ admits exactly $m+1$ different $\alpha$-limit sets.

\section{Lyapunov Function}

\par We first introduce the definition of isolated invariant set, and prove that both proper attracting set and proper repelling set are isolated  invariant sets when they are not intersected. Then we define the weak Morse decomposition of $X$ via leveled $A$-$R$ pair decomposition and obtain a bounded Lyapunov function of $f$.
\begin{definition}[Franks $\&$ Richeson \cite{ref13}]  \rm Let $f\in \mathcal{A}$ then $f:X\rightarrow X$ is a continuous map, for any set $N\subset X$ define \textsl{Inv}$^{m}N$ to be the set of $x\in N$ such that there exists an orbit segment $\{x_{n}\}_{-m}^{m}\subset N$ with $x_{0}=x$ and $f(x_{n})=x_{n+1}$ for $n=-m,\ldots,m-1$. We will call a complete orbit containing $x$ a solution through $x$. More precisely, if $\sigma: \mathbb{Z}\to N$ is given by $\sigma(n)=x_{n}$ and $x_{0}=x$ and $f(x_{n})=x_{n+1}$ for all $n$, we will call $\sigma$ a solution through $x$. We define \textsl{Inv}$N$ to be \textsl{Inv}$^{\infty}N$, the set of $x\in N$ such that there exists a solution $\sigma$ with $\{\sigma(n)\}_{-\infty}^{\infty}\subset N$ and with $\sigma(0)=x$.
\end{definition}
\par Note that from the definition it is clear that $f($\textsl{Inv}$N)=$\textsl{Inv}$N$. The following basic property of \textsl{Inv}$N$ is trivial if $f$ is one-to-one, and also holds for general continuous maps.

\begin{proposition}[Franks $\&$ Richeson \cite{ref13}] \rm If $f$ is a continuous selfmap on compact metric space $X$, $N$ is a compact subset of $X$, then
$$  \textsl{Inv}N=\bigcap_{m=0}^{\infty}\textsl{Inv}^{m}N.
$$
\end{proposition}
\begin{definition}\rm
A compact set $N$ is called an isolating neighborhood if $\textsl{Inv}N\subset\textsl{int}N$. A set $S$ is called an isolated invariant set if there exists an isolating neighborhood $N$ with $S=\textsl{Inv}N$.
\end{definition}
\par  Suppose $m$ be the level of $(X,f)$.  Notice that $A_{i}$ and $R_{i}$ may be intersected, i.e., $A_{i}\cap R_{i}\neq\emptyset$ $(0\leq i\leq m)$, but the interior is always disjoint. At this case, both $A_{i}$ and $R_{i}$ are obviously not isolated invariant sets.

\begin{lemma}\rm
Let $f\in \mathcal{A}$ and $m$ be the level of $(X,f)$ $(0\leq m\leq\infty )$. If $A_{i}\cap R_{i}=\emptyset $ $(1\leq i\leq m)$, then both $A_{i}$ and $R_{i}$ are isolated invariant sets of $f$.
\end{lemma}
\begin{proof} \rm If $A_{i}\cap R_{i}=\emptyset $ $(1\leq i\leq m)$, for any given proper attracting set $A_{i}$$(1\leq i\leq m)$, we consider the map $f_{i}:=f|_{A_{i-1}}$. Denote $F_{i}=\bigcup_{n\geq0}f^{-n}(\textsl{int}(\mathcal{B}(A_{i})))$, we can find an open set $A_{i}\subset U_{i}\subset F_{i}$ such that for any $x\in F_{i}\backslash U_{i}$, there exists positive integer $n(x)$ so that $f_{i}^{n(x)}(x)\in U_{i}$. Put $V_{i}:=X\backslash U_{i}$. Obviously, $V_{i}$ is compact and $R_{i}\subset$ \textsl{int}$(V_{i})$. In what follows we show that
\begin{center}
{$R_{i}=\textsl{Inv}(f,V_{i}):=\{x\in V_{i}|$ there exists a solution  $\{\sigma(n)\}_{-\infty}^{\infty}\subset V_{i}$ with $\sigma(0)=x\}.$
}
\end{center}

\noindent Since $R_{i}$ is a bi-invariant closed set, we have $\{f^{k}(x)\}\subset R_{i}$ for any $x\in R_{i}$ and each $k\in \mathbb{Z}$. So $x\in \textsl{Inv}(f,V_{i})$. As a result, $R_{i}\subseteq \textsl{Inv}(f,V_{i})$. On the other hand, if $x\notin R_{i}$, then the forward orbit $orb(x)\bigcap F_{i}\neq \emptyset$. By the definition of $U_{i}$, we know that $orb(x)\bigcap U_{i}\neq \emptyset$, so when $n$ is big enough, $f^{n}(x)\in U_{i}$, which implies $f^{n}(x)\notin X\backslash U_{i}$, and $x\notin \textsl{Inv}(f,V_{i})$. Hence $\textsl{Inv}(f,V_{i})\subseteq R_{i}$.
\par Next we show attracting set is an isolated invariant set. Given any proper attracting set $A_{i}$  and corresponding $R_{i}$ $(1\leq i\leq m)$, we consider the map $f_{i}:=f|_{A_{i-1}}$. Denote $G_{i}=X\setminus A_{i}$, by the definition of $R_{i}$, we can find an open set $R_{i}\subset Q_{i}\subset G_{i}$ such that for any $x\in Q_{i}\setminus R_{i}$, there exists positive integer $n(x)$ so that $f_{i}^{-n(x)}(x)\subset R_{i}$. Put $P_{i}:=X\setminus Q_{i}$. Obviously, $P_{i}$ is compact and $A_{i}\subset$ \textsl{int}$(P_{i})$. In what follows we show that
\begin{center}
{$A_{i}=\textsl{Inv}(f,P_{i}):=\{x\in P_{i}|$ there exists a solution  $\{\sigma(n)\}_{-\infty}^{\infty}\subset P_{i}$ with $\sigma(0)=x\}.$
}

\end{center}

\noindent Since $A_{i}$ is an invariant closed set, we have $f(A_{i})=A_{i}$ for any $x\in A_{i}$. Hence for each $k\in \mathbb{Z}$, there exists a solution $\{f^{k}(x)\}\subset A_{i})$. So $x\in \textsl{Inv}(f,P_{i})$. As a result, $A_{i}\subseteq \textsl{Inv}(f,P_{i})$. On the other hand, if $x\notin A_{i}$, then the backward orbit $\bigcup_{n\geq0}f^{-n}(x)\cap G_{i}\neq \emptyset$. By the definition of $Q_{i}$, we know that $\bigcup_{n\geq0}f^{-n}(x)\cap Q_{i}\neq \emptyset$, so when $n$ is big enough, $f^{-n}(x)\subseteq Q_{i}$, which implies $f^{-n}(x) \nsubseteq X\setminus Q_{i}$, and $x\notin \textsl{Inv}(f,P_{i})$. Hence $\textsl{Inv}(f,P_{i})\subseteq A_{i}$.  $\blacksquare$
\end{proof}
\par For instance, let $f$ be an $m$-renormalizable topologically expansive Lorenz map or an $m$-renormalizable unimodal map without homtervals. According to Example 2.4, if $m$ is infinite, then both $A_{i}$ and $R_{i}$ are isolated invariant sets for any $1\leq i\leq m$; if $m$ is finite, then both $A_{i}$ and $R_{i}$ are isolated invariant sets for $ 1\leq i\leq m-1$, but $A_{m}$ and $R_{m}$ may be intersected.
 \par Conley put forward the decomposition of isolated invariant set that are consistent with concept of isolating neighborhood and obtained a Lyapunov function, which is called Conley's fundamental decomposition theorem \cite{ref9}. In fact, $X$ is decomposable when $f\in\mathcal{F}$ is not topologically transitive on $X$, but $X$ is not an isolated invariant set since we can not find an isolating neighborhood $N$ such that $X\subset \textsl{int}N$. Hence we slightly modify Conley's fundamental decomposition theorem.
\begin{lemma}\rm
Let $(A,R)$ be an $A$-$R$ pair decomposition of $X$, and $L$ be the set $A\cap R$, then there exists a bounded function $ V:X\setminus L\rightarrow[0,1]$ such that:
\par (1) $V^{-1}(1)=R\setminus L$.
\par (2) $V^{-1}(0)=A\setminus L$.
\par (3) If $x\in C$, then $V(x)>V(f^{n}(x))$ for $n>0$.
\par (4) If $L=\emptyset$, then $V(x)$ is continuous function.

\end{lemma}
\begin{proof}\rm
Observe that if $A$ or $R$ equals the empty set, then the result is trivially true. So assume that $A\neq \emptyset\neq R$. The proof is divided into three steps. The first is to define a function $g:X\setminus L\to [0,1]$ by
$$
g(x)=\frac{d(x,A)}
{d(x,A)+d(x,R)},
$$
where $d(x,A)$ is the distance between $x$ and $A$. We consider $g(x)$ confined on $X\setminus L$ since $g(x)=0/0$ has no meaning for $\forall x\in L$. Clearly, $g^{-1}(0)=A\setminus L$, and $g^{-1}(1)=R\setminus L$.
\par Second, define $h:X\setminus L\rightarrow[0,1]$ by
$$
h(x)=\sup_{n\geq0}\{g(f^{n}(x))\}.
$$
It is obvious that $h^{-1}(0)=A\setminus L$, and $h^{-1}(1)=R\setminus L$ and $h(f^{m}(x))\leq h(x)$ for all $m\geq0$.
\par The third and final step is to define
$$
V(x)=\sum_{m>0}^{\infty}e^{-m}h(f^{m}(x)).
$$
Clearly (1) and (2) are satisfied. If $x\in C$, then
\[V(x)-V(f^{n}(x))=\sum_{m>0}^{\infty}e^{-m}h(f^{m}(x))-\sum_{m>0}^{\infty}e^{-m}h(f^{m+n}(x))\]
\[=\sum_{m>0}^{\infty}e^{-m}[h(f^{m}(x))-h(f^{m+n}(x))]\]
 \[\geq0.\]

Now $V(x)-V(f^{n}(x))=0$ if and only if $h(f^{m}(x))-h(f^{m+n}(x))=0$ for all $m\geq0$, i.e., $h(f^{m}(x))=\bar{h}$ a constant for all $m\geq0$. This implies that $\omega(x)\cap (A\cup R)=\emptyset$, which contradicts to the definition of connecting orbits. Thus $V(x)-V(f^{n}(x))>0$, and conclusion (3) holds. When $A\cap R=\emptyset$, which implies $L=\emptyset$ and $V(x)$ is defined on the whole space $X$. It is clear that $g(x)$ and $h(x)$ is continuous, hence $V(x)$ is a continuous function, conclusion (4) is verified. The proof is completed. $\hfill\square$
\end{proof}
\par Conley \cite{ref9} extented fundamental decomposition theorem to Morse decomposition via Morse sets, and emphasized that Morse sets are disjoint compact invariant sets. However, when $A_{i}$ and $R_{i}$ $(0\leq i\leq m)$ are intersected, $A_{i}$ is clearly not an isolated invariant set, but $A_{i}$ may still be decomposable. Hence we slightly modify the definition of Morse sets and obtain the weak Morse decomposition of $X$ via leveled $A$-$R$ decomposition. Observe that weak Morse decomposition of $X$ is more delicate, and it is almost the same as Morse decomposition of isolated invariant set when $A_{i}\cap R_{i}=\emptyset$ for any $0\leq i\leq m$.
\begin{definition} \rm A Morse decomposition of $X$ is a finite collection of interior disjoint compact invariant sets called weak Morse sets, $$ \mathcal{M}(X):=\{M_{p}\mid p\in\mathcal{P}\}$$
satisfying two properties. First, for any $x\in X$, both $\alpha(x)$ and $\omega(x)$ of $x$ lie in the weak Morse sets. Second, there exists a strict partial ordering on the indexing set $\mathcal{P}$ such that if $x\in X\setminus \bigcup M_{p}$, $\alpha(x)\subset M_{p}$, and $\omega(x)\subset M_{q}$, then $p>q$.
\end{definition}
\par If the level of $(X,f)$ is $m$, then the leveled $A$-$R$ decomposition of $X$ is $\{(A_{1},R_{1}), (A_{2},R_{2}),\ldots$,
$(A_{m},R_{m})\}$. Denote $M_{i}:=A_{i}\cap R_{i+1}$ and $L_{i}=A_{i}\cap R_{i}$ $(0\leqslant i\leqslant m)$, set $R_{m+1}=X$, then all interior disjoint $M_{i}$ are clearly weak Morse sets, $M_{D}=\bigcup_{i=0}^{m}M_{i}$ is called the Morse decomposition of $X$. As with attracting-repelling pair decomposition, Morse decomposition of $X$ also admit Lyapunov function.

\par {\bf Theorem B} \ \ \rm Let $M_{i}:=A_{i}\cap R_{i+1}(0\leqslant i\leqslant m)$, $L:=\bigcup_{i=0}^{m}L_{i}$ and $C_{j}:=C(M_{j+1},M_{j})$ $(0\leq j\leq m-1)$ be the connecting orbits. Then $M_{D}=\bigcup_{i=0}^{m}M_{i}$ is a weak Morse decomposition of $X$, and there exists a bounded function $V:X\setminus L\to[0,1]$ such that:
\par (1) $\forall$ $x,y\in M_{i}\setminus L$, $0\leqslant i\leqslant m$, $V(x)=V(y)$.
\par (2) $V^{-1}(0)=V^{-1}(1/2^{m})=M_{m}\setminus L$.
\par (3) $V^{-1}(1)=M_{0}\setminus L$, $V^{-1}(1/2)=M_{1}\setminus L$, $\cdots, V^{-1}(1/2^{m-1})=M_{m-1}\setminus L$.
\par (4) If $x\in X\setminus M_{D}$, then $V(x)>V(f^{n}(x))$ for $n>0$.
\par (5) If $L=\emptyset$, then $V(x)$ is continuous.
\par (6) $\forall x\in X\setminus L$ and $0\leq j\leq m-1$. If $V(x)=0$, $\alpha(x)=X$; if $V(x)=1/2^j$, $\alpha(x)=R_{j+1}$; if $V(x)\in(1/2^{j+1},1/2^{j})$, then $x\in C_{j}$ and $\alpha(x)=R_{j+1}$.

\begin{proof}\rm
Observe that if $m=0$, then $M_{D}=\{M_{0}=X\}$ is the simplest Morse decomposition. When $m=1$, the result $M_{D}=\{M_{0}=R_{1},M_{1}=A_{1}\}$ is the same as Lemma 4.5. Considering $m>1$, notice that if $m$ is infinite, $M_{m}=A_{m}=\bigcap_{i=0}^{m}A_{i}$, the construction of $V(x)$ is divided into three steps. The first is to define a function $g:X\setminus L\rightarrow[0,1]$ by
$$
g(x)=\sum_{i=0}^{m-1}\frac{d(x,M_{0})\cdots d(x,M_{i-1})d(x,M_{i+1})\cdots d(x,M_{m})}
{d(x,M_{i})+2^{m}d(x,M_{0})\cdots d(x,M_{i-1})d(x,M_{i+1})\cdots d(x,M_{m})},
$$
where $d(x,M_{i})$ is the distance between $x$ and $M_{i}$, $0\leqslant i\leqslant m$. We consider $g(x)$ confined on $X\setminus L$ since $g(x)=0/0$ has no meaning on $L$. Clearly, $g^{-1}(0)=M_{m}\setminus L$, and $g^{-1}(1/2^{i})=M_{i}\setminus L$. Second, define $h:X\setminus L\rightarrow[0,1]$ by
$$
h(x)=\sup_{n\geq0}\{g(f^{n}(x))\},
$$
it is obvious that $h^{-1}(0)=M_{m}\setminus L$,  $h^{-1}(1/2^{i})=M_{i}\setminus L$ and $h(f^{m}(x))\leq h(x)$ for all $m\geq0$. The third and final step is to define
$$
V(x)=\sum_{m>0}^{\infty}e^{-m}h(f^{m}(x)).
$$
Clearly (1), (2) and (3) are satisfied. If $x\in X\setminus M_{D}$, then
\[V(x)-V(f^{n}(x))=\sum_{m>0}^{\infty}e^{-m}h(f^{m}(x))-\sum_{m>0}^{\infty}e^{-m}h(f^{m+n}(x))\]
\[=\sum_{m>0}^{\infty}e^{-m}[h(f^{m}(x))-h(f^{m+n}(x))]\]
 \[\geq0.\]

Now $V(x)-V(f^{n}(x))=0$ if and only if $h(f^{m}(x))-h(f^{m+n}(x))=0$ for all $m\geq0$, i.e., $h(f^{m}(x))=\bar{h}$ a constant for all $m\geq0$. This implies that $\omega(x)\cap M_{D}=\emptyset$, a contradiction with the definition of connecting orbits. Thus $V(x)-V(f^{n}(x))>0$, and conclusion (4) holds. When $A_{i}\cap R_{i}=\emptyset$ ($0\leqslant i\leqslant m$), which implies $L=\emptyset$ and $V(x)$ is defined on $X$. It is clear that $g(x)$ and $h(x)$ is continuous, hence $V(x)$ is a continuous Lyapunov function, conclusion (5) is verified.
\par As for (6), given any $V(x)\in[0,1]$, we can clearly see $x$ belongs to which $M_{i}$ or which connecting orbit, and obtain its $\alpha$-limit set via $V(x)$. If $V(x)=0$, then $x$ belongs to the topological attractor $A_{m}$, and $\alpha(x)=X$; if $V(x)=1$, $x$ belongs to the first repeller $R_{1}$, i.e., $x\in M_{0}$ and $\alpha(x)=R_{1}$. Furthermore, all the points in $X\setminus M_{D}$ are divided into $m$ different connecting orbits. If $V(x)\in(1/2^{i},1/2^{i-1})$, $1\leq i\leq m$, it implies $x$ belongs to the connecting orbit $C(M_{i},M_{i-1})$ and $\alpha(x)=E_{i}$.
Hence we can study the orbit behavior of each point via $V(x)$.
\par As for (6), $\forall x\in X\setminus L$, we can obtain a value $V(x)\in [0,1]$. If $V(x)=0$, $x\in M_{m}\setminus L=A_{m}\setminus L$ and $x$ belongs to the attractor, then according to Theorem A, $\alpha(x)=X$. If $V(x)=1/2^j$ $(0\leq j\leq m-1)$, then $x\in M_{j}\setminus L$, according to Theorem A, $\alpha(x)=R_{i+1}$. Furthermore, all the points in $X\setminus M_{D}$ are divided into $m$ different connecting orbits. If $V(x)\in(1/2^{j+1},1/2^{j})$, $0\leq j\leq m-1$, then $x\in C_{j}$. By the definition of connecting orbit, $\alpha(x)=R_{i+1}$.
Hence we can characterize the $\alpha$-limit set of each point in $X\setminus L$ via $V(x)$.   $\hfill\square$
\end{proof}

\par

Let $f$ be an $m$-renormalizable expansive Lorenz map or an $m$-renormalizable unimodal map without homtervals. If $m$ is finite, $A_{m}$ is the topological attractor, and $f|_{A_{m}}$ is topologically transitive on $A_{m}$.  According to Example 2.4, there exists possibility such that $A_{m}\cap R_{m}\neq\emptyset$, but the interior is always disjoint. If $A_{m}\cap R_{m}\neq\emptyset$, then $V(x)$ of $f$ is not continuous; if $A_{m}\cap R_{m}=\emptyset$, then $V(x)$ of $f$ is continuous.

\section*{Acknowledgements}{\rm The author thanks the referees for reading this paper very carefully. They suggested many improvements to a previous version of this paper. This work is supported by the Excellent Dissertation Cultivation Funds of Wuhan University of Technology (2018-YS-077)


\section*{References}
\begin{thebibliography}{99}


\bibitem{ref3}
 Auslander J, Bhatia N P, Seibert P. Attractors in dynamical systems. Bol Soc Mat Mex, 1964, \textbf{9}: 55-66
 \bibitem{ref22}Batko B, Mrozek M. Weak index pairs and the Conley index for discrete multivalued dynamical systems. SIAM J Appl Dyn Syst, 2016, \textbf{15}(2): 1143-1162
 \bibitem{ref17}Bruin H, Keller G, Nowicki T, van Strien S. Wild Cantor attractors exist. Ann Math, 1996, \textbf{143}(1): 97-130

 \bibitem{ref9}Conley C. Isolated invariant sets and the Morse index//Regional Conference Series in Mathematics vol 38. Am Math Soc, 1978

\bibitem{ref1}
 Coddington E A, Levinson N. Theory of ordinary differential equations. Tata Mcgraw-Hill Education, 1955

 \bibitem{ref7} Collet P, Eckmann J P. Concepts and results in chaotic dynamics: a short course//Theoretical and Mathematical Physics. Berlin: Springer-Verlag, 2006


 \bibitem{ref12}Cui H F, Ding Y M. The $\alpha$-limit sets of a unimodal map without homtervals. Topology Appl, 2010, \textbf{157}(1): 22-28

\bibitem{ref27}Cui H F, Ding Y M. Renormalizaiton and conjugacy of piecewise linear Lorenz maps. Adv Math, 2015, \textbf{271}: 235-272

 \bibitem{ref11}Ding Y M. Renormalziation and $\alpha$-limit set of expanding Lorenz map. Discrete Contin Dyn Syst, 2011, \textbf{29}(3): 979-999

\bibitem{ref26}Ding Y M, Fan W T. The asymptotic periodicity of Lorenz maps. Acta Math Sci, 1999, \textbf{19}(1): 114-120
 \bibitem{ref13}Franks J, Richeson D. Shift equivalence and the Conley index. Trans Amer Math Soc, 2000, \textbf{352}(7): 3305-3322

\bibitem{ref19} Good C, Meddaugh J, Mitchell J. Shadowing, internal chain transitivity and $\alpha$-limit sets. J Math Anal Appl, 2020, \textbf{491}(1): 124291, 19 pp
\bibitem{ref8} Guckenheimer J, Holmes P. Nonlinear oscillations, dyanmcial systems, and bifurcations of vector fields// Appl Math Sci vol 42. New York: Springer -Verlag, 1983

\bibitem{ref18} Graczyk J, Kozlovski O. On Hausdorff dimension of unimodal atttractors. Commun Math Phys, 2006, \textbf{264}(3): 565-581
\bibitem{ref25} Hubbard J, Sparrow C. The classification of topologically expansive Lorenz maps. Comm Pure Appl Math, 1990, \textbf{43}(4): 431-443

\bibitem{ref23} Kaczynski T, Mrozek M. Conley index for discrete multi-valued dynamical systems. Topology Appl, 1995, \textbf{65}(1): 83-96

     \bibitem{ref16}Li S M, Shen W X. Hausdorff dimension of Cantor attractors in one-dimensional dynamics. Invent Math, 2008, \textbf{171}(2): 345-387
     \bibitem{ref20}Liu Z X. Conley index for random dynamical systems. J Diffrential Equations, 2008, \textbf{244}(7): 1603-1628

 \bibitem{ref6} Milnor J. On the concepts of attractor//The Theory of Chaotic Attractors. New York: Springer, 1985, 243-264

 \bibitem{ref14}Mischaikow K. The conley index theory: a brief introduction. Banach Center Publ, 1999, \textbf{47}(1): 9-19

\bibitem{ref2} Mendelosn P. On unstable attractors. Bol Soc Mat Mex, 1960, \textbf{5}: 270-276

\bibitem{ref10}Mane R. Hyperbolicity, sinks and measure in one-dimensional dynamics. Comm Math Phys, 1985, \textbf{100}(4): 495-524

\bibitem{ref4} Smale S. Differential dynamical systems. Bull Amer Math Soc, 1967,  \textbf{73}: 747-817

\bibitem{ref28} Wang J T. On the theory of Conley index and shape Conley index in general metric spaces[Ph D Thesis]. Tianjin: Tianjin University, 2016
\bibitem{ref5} Williams R F. The zeta function of an atrractor//Conference on the Topology of Manifolds. Boston: Prindle Weber and Schmidt, 1968: 155-161

















\end{thebibliography}
\end{document}